\documentclass[12pt]{article}
\usepackage{amsmath}
\usepackage{amsfonts}
\usepackage{amssymb}
\usepackage{amsthm}
\usepackage{stackrel}

\usepackage{color}

\usepackage{figsize}

\usepackage{graphics}

\newtheorem{thm}{Theorem}[section]

\newcommand{\R}{{\rm I}\kern-0.18em{\rm R}}
\newcommand{\1}{{\rm 1}\kern-0.25em{\rm I}}
\newcommand{\E}{{\rm I}\kern-0.18em{\rm E}}
\newcommand{\p}{{\rm I}\kern-0.18em{\rm P}}

\makeatletter

\def\@fnsymbol#1{\ensuremath{\ifcase#1\or a\or b\or c\or d\or \e\or f\or *\dagger 	\or \ddagger\ddagger \else\@ctrerr\fi}}
\title{Inequalities for m-Divisible Distributions and Testing of Infinite Divisibility}
\author{Lev B. Klebanov\footnote{Department of Probability and Mathematical Statistics, Charles University, Prague, Czech Republic},\; Ashot V. Kakosyan\footnote{Yerevan State University, Yerevan, Armenia}\; and Irina V. Volchenkova\footnote{Czech Technical University in Prague, Czech Republic} }
\date{}
\begin{document}
\maketitle

\begin{abstract}
We state some inequalities for m-divisible and infinite divisible characteristic functions. Basing on them we propose statistical test for a distribution to be infinite divisible.

\end{abstract}

\section{Inequalities for characteristic functions. Estimates from below}\label{sec1} 
\setcounter{equation}{0} 
Let $f(t)$ be characteristic function of a distribution on real line. We say that $f(t)$ is $m$-divisible ($m$ is positive integer) if $f^{1/m}(t)$ is a characteristic function as well. The function $f(t)$ is infinitely divisible if it is $m$-divisible for all positive integers $m$. Properties of infinitely divisible characteristic functions were well-studied (see, for example, \cite{L}). In particular, any probability distribution with a compact support is not infinitely divisible. However, it is not true for $m$-divisible distributions. Really, let us take arbitrary distribution with a compact support. Denote its characteristic function by $g(t)$ and define $f(t)=g^m(t)$. Clearly, $f(t)$ is $m$-divisible characteristic function of a distribution with compact support.

Let us start with the case of arbitrary symmetric distribution having compact support. 

\begin{thm}\label{th1}
Let $g(t)$ be a characteristic function of an even non-degenerate distribution function having its support in interval $[-A,A]$, $A>0$. Then the inequality
\begin{equation}\label{eq1}
\cos(\sigma \cdot t) \leq g(t)
\end{equation}
holds for all $t \in (-4.49/A,4.49/A)$. Here $\sigma$ is standard deviation of the distribution with characteristic function $g(t)$.
\end{thm}
\begin{proof}
Denote by $G(x)$ a distribution function corresponding to the characteristic function $g(t)$. We have 
\begin{equation}\label{eq2}
 g(t)= \int_{-A}^{A} \cos(tx) dG(x) = \int_{0}^{A}\cos(t x)d\tilde{G}(x)=\int_{0}^{A^2}\cos(t \sqrt{y})dH(y), 
\end{equation}
where $\tilde{G}(x)=2(G(x)-1/2)$ and $H(y)=\tilde{G}(\sqrt{y})$. In view of the fact that $\cos(t)$ is an even function we can consider positive values of $t$ only. 

It is not difficult to see that the function $\cos(t\sqrt{y})$ is convex in $y$ for the case when $0 \leq t \sqrt{y} \leq z_o$, where $z_o$ is the first positive root of the equation $\sin z- z\cos z=0$. Numerical calculations show that $z_o >4.49$. For the case of $|t|\leq 4.49/A$ let us apply Jensen inequality to (\ref{eq2}). We obtain
\[g(t)=\int_{0}^{A^2}\cos(t \sqrt{y})dH(y) \geq \cos\Bigl(t \sqrt{\int_{0}^{A^2}ydH(y)\;}\Bigr)=\cos(\sigma t). \]
\end{proof}

The condition of support compactness may be changed by the restriction of absolute fifth moment existence. Namely, the following result holds.

\begin{thm}\label{th1a}
Let $G(x)$ be a symmetric probability distribution function. Denote by $a_j$ its absolute moments and suppose that $a_{10}$ is finite.
Denote by $g(t)$ corresponding characteristic function and $\sigma^2 =a_2$. Then we have
\begin{equation}\label{eq2a}
\cos(\sigma t) \leq g(t)
\end{equation}
for all $|t| \leq 5(a_4-\sigma^4)/( \sigma^{10}+2\sigma^5 a_5+ a_{10})^{1/2} $.
\end{thm} 
\begin{proof} 
Suppose that $g(t)$ is not identical to $\cos(\sigma t)$. Define
\[\varphi(t)=g(t)-\cos(\sigma t).\]
It is easy to verify that $\varphi(t)$ and its derivatives at the point $t=0$ satisfy
\[\varphi^{(k)}(0)=0,\quad \text{for}\quad k=0,1,2,3,  \]
and $\varphi^{(4)}(0)=a_4-a^2_2 >0$. Therefore, $\varphi^{(4)}(t)$ is positive in some neighborhood of the point $t=0$ and, consequently,
$\varphi(t)$ is non-negative at least in some neighborhood obtained by means of forth times integration. Our aim now is to estimate the length of the interval for $\varphi^{(4)}(t)$ positiveness. For this consider derivative of $\varphi^{(4)}(t)$ that is $\varphi^{(5)}(t)$. We have
\[ |\varphi^{(5)}(t)| = |\int_{-\infty}^{\infty}\bigl(\sigma^5 \sin(\sigma t)-  \sin(t x)x^5 \bigr) dG(x)| \leq \] 
\[ \leq \Bigl( \int_{-\infty}^{\infty}\bigl(\sigma^5 \sin(\sigma t)-  \sin(t x)x^5 \bigr)^2 dG(x) \Bigr)^{1/2} \leq \Bigl( \sigma^{10}+2\sigma^5 a_5+ a_{10}\Bigr)^{1/2}.\] 
In view of the facts that $\varphi^{(4)}(0)= a_4-\sigma^4>0$ we see that $\varphi^{(4)}(t) \geq 0$ on the interval $[0,(a_4-\sigma^4)/( \sigma^{10}+2\sigma^5 a_5+ a_{10})^{1/2}]$, and, because of symmetry, on interval $-(a_4-\sigma^4)/( \sigma^{10}+2\sigma^5 a_5+ a_{10})^{1/2},(a_4-\sigma^4)/( \sigma^{10}+2\sigma^5 a_5+ a_{10})^{1/2}$. This guarantees that 
\[\varphi(t) \geq \frac{t^4}{24}\Bigl((a_4-\sigma^4)-( \sigma^{10}+2\sigma^5 a_5+ a_{10})^{1/2}t/5 \Bigr). \] 
From this and the symmetry of $\varphi(t)$ follows the result.
\end{proof}

Let us turn to the of $m$-divisible distribution with compact support.

\begin{thm}\label{th2}
Let $f(t)$ be a characteristic function of $m$-divisible symmetric distribution having compact support in $[-A,A]$, $A>0$ and standard deviation $\sigma$. Then
\begin{equation}\label{eq3}
\cos^m(\sigma t/\sqrt{m}) \leq f(t)
\end{equation}
for $|t|\leq \min(4.49 m/A,\pi \sqrt{m}/(2\sigma))$. 
\end{thm}
\begin{proof}
Denote $g(t)=f^{1/m}(t)$. It is clear that:
\begin{enumerate}
\item $g(t)$ is a symmetric characteristic function;
\item $m\sigma^2(g)=\sigma^2(f)=\sigma^2$ (because variance of sum of independent random variables equals to the sum of their variances);
\item distribution with characteristic function $g$ has compact support in \\ $[-A/m,A/m]$ (see, for example, \cite{LO}, Theorem 3.2.1).
\end{enumerate}
Applying Theorem \ref{th1} to the function $g(t)$ we find
\[\cos(\sigma t/\sqrt{m}) \leq g(t) = f^{1/m}(t)\]
for $|t| \leq 4.49m/A$. However, for $|t|\leq \min(4.49 m/A,\pi \sqrt{m}/(2\sigma))$ the left hand side of previous inequality is non-negative and we come to the conclusions of Theorem \ref{th2}.
\end{proof}
Note that $\cos(\sigma t/\sqrt{m})$ is monotone increasing in $m$ for $|t|\leq \pi\sqrt{m}/(2\sigma)$ and, therefore the estimator (\ref{eq3}) is more precise than (\ref{eq1}).

Let us give a little bit different result.
\begin{thm}\label{th2a}
Let $f(t)$ be a characteristic function of $m$-divisible symmetric distribution having finite tenth moment $a_{10}$. Then
\begin{equation}\label{eq3}
\cos^m(\sigma t/\sqrt{m}) \leq f(t)
\end{equation}
for $|t|\leq \min(C\sqrt{m},\pi \sqrt{m}/(2\sigma))$, where positive $C$ depends on absolute moments $a_k,\;(k=1,\ldots ,10)$ only. 
\end{thm}
\begin{proof} Denote $g(t)=f^{1/m}(t)$. It is clear that:
\begin{enumerate}
\item $g(t)$ is a symmetric characteristic function;
\item $m\sigma^2(g)=\sigma^2(f)=\sigma^2$ (because variance of sum of independent random variables equals to the sum of their variances);
\item distribution with characteristic function $g$ has finite absolute moments up to tenth order (see, for example, \cite{LO}). 
\end{enumerate}
It is not difficult to verify that 
\begin{equation}\label{eq4}
a_{k,m} \sim a_k/m \quad \text{as} \quad m\to \infty,
\end{equation}
where $a_{k,m}$ is $k$th absolute moment of $g$.
To finish the proof it is enough to apply Theorem \ref{th1a} and the relation (\ref{eq4}). 
\end{proof}

Consider now the case of infinitely divisible distribution.

\begin{thm}\label{th3}
Let $f(t)$ be a symmetric infinite divisible characteristic function with finite second moment $\sigma^2$. Then
\begin{equation}\label{eq5}
\exp\{-\sigma^2 t^2/2\} \leq f(t)
\end{equation}
for all $t \in \R^1$.
\end{thm}
\begin{proof}
If $f(t)$ has finite tenth moment it is sufficient pass to limit in (\ref{eq3}) as $m\to \infty$. In general case one can approximate $f$ by infinitely divisible characteristic functions with finite tenth moment.
\end{proof}
\begin{proof}[Another proof] Kolmogorov representation formula (see, for example \cite{L}) for $f(t)$ allows us to rewrite (\ref{eq5}) in the form
\[\exp \{-\sigma^2 t^2/2\}\leq \exp\{-\int_{-\infty}^{\infty}\bigl(1-\cos (t x)\bigr)/x^2 d K(x)\} \]
or, equivalently,
\begin{equation}\label{eq6}
\int_{-\infty}^{\infty}\bigl(1-\cos (t x)\bigr)/x^2 d K (x) \leq \sigma^2 t^2/2.
\end{equation}
However,
\[2 \bigl(1-\cos (t x)\bigr)/x^2 \leq 4 \sin^2(\frac{tx}{2})/x^2 \leq t^2 .\]
This leads to (\ref{eq6}) with 
\[\sigma^2 = \int_{-\infty}^{\infty}dK(x)=-f^{\prime \prime}(0).\]
\end{proof}
From the last proof it follows that {\it if the equality in (\ref{eq5}) attends in a point $t_o \neq 0$ then it holds for all $t \in \R^1$}.

Theorem \ref{th3} shows an extreme property of Gaussian distribution among the class of infinite divisible distributions with finite second moment. Another extreme property without any moment conditions was given in \cite{KKR}. Let us give this result here.

\begin{thm}\label{th4}
Let $f(t)$ be a symmetric infinite divisible characteristic function. Then
\begin{equation}\label{eq7}
f(t) \geq f^4(t/2)
\end{equation}
for all $t \in R^1$.  If the equality in (\ref{eq7}) attends in a point $t_o \neq 0$ then it holds for all $t \in \R^1$
\end{thm}  
\begin{proof}
From L\'{e}vy-Khinchin representation we have 
\[f(t) = \exp\{-\int_{-\infty}^{\infty}\bigl(1-\cos(t x)\bigr)\frac{1+x^2}{x^2}d\Theta (x)\}, \]
\[ f^4(t/2) = \exp\{-4\int_{-\infty}^{\infty}\bigl(1-\cos(t x/2)\bigr)\frac{1+x^2}{x^2}d\Theta (x)\}.\]
However,
\[ 1-\cos(t x)= 2 \sin^2(t x/2) =8 \sin^2(t x/4)\cos^2(t x/4) \leq\]
\[\leq  8 \sin^2 (t x/4) =4 \bigl(1-\cos(t x/2)\bigr) \]
\end{proof}

Let us note that Theorem \ref{th3} may be obtained from Theorem \ref{th4}. Really, assuming the existence of finite second moment we have
\[ f(t)\geq f^{4}(t/2) \geq f^{4^2}(t/2^2) \geq \ldots \geq f^{4^k}(t/2^k) \to \exp\{-\sigma^2 t^2/2\}\]
as $k\to \infty$. This proves the inequality (\ref{eq5}).

\section{Inequalities for characteristic functions. Estimates from above}\label{sec2} 

Our aim here is to proof the following result. 

\begin{thm}\label{th2.1}
Let $g(t)$ be a characteristic function of an even non-degenerate function having its support in interval $[-A,A]$, $A>0$. Then the inequality
\begin{equation}\label{eq2.1}
g(t) \leq \cos(a_{1/\gamma}^{\gamma}\cdot t) 
\end{equation}
holds for all $t \in (-\pi/(2A^{\gamma}),\pi/(2A^{\gamma}))$. Here $a_{1/\gamma}$ is absolute moment of the order $1/\gamma$ of the distribution with characteristic function $g(t)$, and $\gamma >1$.
\end{thm}
\begin{proof}
Denote by $G(x)$ probability distribution function corresponding to characteristic function $g(t)$. Set $\tilde{G}(x)=2\bigl(G(xk)-1/2\bigr)$,
$H(y)=\tilde{G}(y^{\gamma})$. We have
\[ g(t) = \int_{-A}^{A}\cos (t x)dG(x) = \int_{0}^{A} \cos(t x)d\tilde{G}(x) = \]
\[ =\int_{0}^{A}\cos (t y^{\gamma})dH(y) \leq \cos ( a_{1/\gamma}^{\gamma}\cdot t) \]
for all $|t| \leq \pi/(2 A^{\gamma})$. Here we used the fact that the function $\cos (t y^{\gamma})$ is concave in $y$ if $0 \leq y |t|\leq \pi/2$ and applied Jensen inequality.
\end{proof}

\section{Inequalities for some moments of infinitely divisible distributions}\label{sec3}
\setcounter{equation}{0} 

Let us give some inequalities comparing the moments of infinitely divisible distributions with corresponding characteristics of Gaussian distribution.

\begin{thm}\label{th3.1}
Let $X$ be a random variable having symmetric infinitely divisible distribution with finite second moment. Suppose that $0<r<2$. Then
\begin{equation}\label{eq2.1}
\E |X|^r \leq \E |Y|^r=\frac{2^{r/2}\sigma^r \Gamma((1+r)/2)}{\sqrt{\pi}},
\end{equation}
where random variable $Y$ has symmetric Gaussian distribution with the same second moment $\sigma^2$ as $X$. The equality in (\ref{eq2.1}) attends if and only if $X$ has Gaussian distribution.  
\end{thm}
\begin{proof}
Recall that if $Z$ is a random variable with characteristic function $h(t)$ then 
\[\E|Z|^r =C_r \int_{0}^{\infty}\frac{1-Re(h(t))}{t^{r+1}}dt,\]
where $C_r$ depends on $r$ only ($0<r<2$). From (\ref{eq5}) it follows that
\[ 1-\exp\{-\sigma^2 t^2/2\} \geq 1-f(t)  \]
and
\[C_r\int_{0}^{\infty}\frac{1-\exp(-\sigma^2 t^2/2)}{t^{r+1}}dt \geq  C_r\int_{0}^{\infty}\frac{1-f(t)}{t^{r+1}}dt. \]
\end{proof}

\section*{Acknowledgement}

The study was partially supported by grant GA\v{C}R 19-04412S (Lev Klebanov) and by grant  
SGS18/065/OHK4/1T/13 Czech Technical University in Prague (Irina Volchenkova).

\end{document}